\def\eq#1{Eq.~(\ref{#1})}
\def\onehalf{{\tfrac{1}{2}}}
\def\Ber{{\textrm{Ber}}}

\def\Beta{{\textrm{Beta}}}

\documentclass{amsart}
\usepackage{amsmath,amsfonts,amsthm,amsrefs}

\begin{document}

\theoremstyle{plain}
\newtheorem*{thm}{Theorem}
\newtheorem*{cor}{Corollary} 

\theoremstyle{definition} 
\newtheorem{defn}{Definition}[section] 
\newtheorem{conj}{Conjecture}[section] 
\newtheorem{exmp}{Example}[section]
\theoremstyle{remark} 
\newtheorem*{rem}{Remark} 
\newtheorem*{note}{Note} 
\newtheorem{case}{Case}

\title{The equalization probability of the P\'olya urn}
\author{Timothy C. Wallstrom} 
\address{Theoretical Division, Los Alamos National Laboratory, Los Alamos, NM 87545}
\email{tcw@lanl.gov} 
\begin{abstract}
  We consider a P\'olya urn, started with $b$ black and $w$ white
  balls, where $b>w$. We compute the probability that there are ever
  the same number of black and white balls in the urn, and show that
  it is twice the probability of getting no more than $w-1$ heads in
  $b+w-1$ tosses of a fair coin.
\end{abstract}
\maketitle

An urn contains $b$ black and $w$ white balls, where $b>w$. A ball is
drawn from the urn at random, and then replaced with two balls of the
same color. This same procedure is then repeated indefinitely. What is
the probability that the urn ever contains the same number of black
and white balls?

The problem involves the famous P\'olya
urn~\cite{Eggenberger:1923kv,Polya:1930wp,johnson_urn_1977,mahmoud_polya_2008}. The
solution is not obvious, because the probability of a black or white
draw is constantly changing, and an infinite number of different draw
sequences can lead to equalization.  The purpose of this note is to
show that there is a remarkably simple solution: the equalization
probability is just twice the probability that in $b+w-1$ tosses of a
fair coin, no more than $w-1$ will be heads.

A probabilist would solve this problem by noting that the draws of the
P\'olya urn are \textit{exchangeable}, in the sense that the
probability of drawing any finite sequence of black and white balls
depends only on the total number of black balls, and not on the order
in which they are drawn. She would then invoke de Finetti's theorem,
which states that an exchangeable process is a mixture of independent
and identically distributed (i.i.d.)  processes, which in this case
means Bernoulli processes, or biased random
walks~\cite{kallenberg_foundations_2010,aldous_ecole_1985,feller_introduction_1971}. In
this way, she would reduce the equalization problem for the P\'olya
urn to the gambler's ruin problem, whose well-known solution dates to
the inception of probability
theory~\cite{feller_introduction_1968}. Such an approach is not
without its charms, and also leads to efficient solutions of more
difficult problems, such as the probability that there are ever $k$
more white than black balls. However, it depends on the gambler's ruin
results, and it seems worthwhile, if possible, to prove the result
directly.

In this note, I provide an elementary proof of the main result, which
does not depend on the gambler's ruin results.  This work was inspired
by a recent paper of Antal, Ben-Naim, and
Krapivsky~\cite{Antal:2010uh}, who posed the equalization problem
while working in the context of first-passage theory, and provided a
closed form expression for the solution. The expressions in the
current paper, however, are new.

Let $S_n=B_n-W_n$, the excess of black over white balls after $n$
draws. $S_n$ is a Markov process that starts at $S_0=b-w>0$, and at
each step either increases by one, if a black ball is drawn, or
decreases by one, if a white ball is drawn.  The probabilities of
these two events are just $B_n/N_n$ and $W_n/N_n$, respectively, where
$N_n=b+w+n$ is the number of balls in the urn after $n$ draws.  The
trajectory of $S_n$ thus resembles a random walk, except that the
probabilities of moving up and down change with every step. We are
interested in the probability that the path ever touches the $S$-axis.

We will need one key fact about the process $S_n$, or equivalently,
about $B_n$: the fraction of black balls, $B_n/N_n$, has a random
limit $Z$, which is given by the $\Beta_{b,w}$ distribution:
\begin{equation}
  \label{eq:beta}
  \Beta_{b,w}(p) = \frac{\Gamma(b+w)}{\Gamma(b)\Gamma(w)}\,p^{b-1}(1-p)^{w-1}\quad (0<p<1),
\end{equation}
An elementary proof may be obtained by following the reasoning in de
Finetti~\cite[p. 219]{finetti_theory_1975}, who shows how the P\'olya
urn process may be modeled using random draws from the uniform
distribution. For a less elementary proof,
see~\cite[\S2]{Freedman:1965tv}. It follows that $\mu_n=S_n/N_n$ also
converges, to $\mu\equiv 2Z-1$.  We write
$F^\Beta_{b,w}(p)=\int_0^p\Beta_{b,w}(p)\,dp$ for the distribution
function of the beta distribution.

\begin{thm} If a P\'olya urn is started with $b$ black and $w$ white
  balls, then the probability that the number of black and white
  balls will ever be equal is $2 F^\Beta_{b,w}(\onehalf)$.
\end{thm}

\begin{proof}
  Let $\tau$ be the random time at which the path first touches the
  boundary $S=0$.  The probability of equalization is just the
  probability of the event $\{\tau<\infty\}$, which can be divided
  into the two events $\{\tau<\infty\}\cap\{\mu>0\}$ and
  $\{\tau<\infty\}\cap\{\mu<0\}$. (The probability that $\mu=0$ is
  zero, because the density in~\eq{eq:beta} is continuous.) But these
  two events have equal probability.  Indeed, $\mu$ depends only on
  $S_\tau(n)\equiv S(\tau+n)$, because the initial segment is finite,
  and has no effect on the mean. At time $\tau$, the urn has an equal
  number of black and white balls, so the mean of its subsequent
  trajectory is equally likely to be positive or negative.
  Furthermore, if $\mu<0$, then $\tau<\infty$. Indeed, if $\mu<0$,
  then the path will eventually be below the axis, and since it
  started out above the axis, it must cross at some point. Thus,
 \begin{displaymath}
   P(\tau <\infty) = 2 P (\{\tau<\infty\}\cap\{\mu<0\}) = 
   2P(\{\mu<0\}) = 2F^\Beta_{b,w}(\onehalf).
 \end{displaymath}
 The last expression follows from the fact that $\mu<0$ if
 and only if $p<\onehalf$.
\end{proof}

The equalization probability can also be expressed as a binomial sum,
due to an interesting connection between the beta and uniform
distributions.  Let $U_1,U_2\ldots,U_n$ be $n$ independent samples
from the uniform distribution on $[0,1]$, and let
$U_{(1)}<U_{(2)}<\cdots<U_{(n)}$ be the same samples arranged in
increasing order. (The $U_{(i)}$ are called the \textit{order
  statistics} of the sample.) Then the density of $U_{(b)}$ is given
by $\Beta_{b,w}$, where $b+w=n+1$. We refer the reader
to~\cite[Eq. I.6.7]{feller_introduction_1971} for
the simple proof.

Given this result, the probability that a $\Beta_{b,w}$ variable will
be less than $x$ is just the probability that at least $b$ uniform
variates will be less than $x$, or equivalently, that no more than
$n-b$ uniform variates will be greater than $x$.  Let $\Ber_p$ denote
a Bernoulli random variable, taking the value one with probabibility
$p$, and zero otherwise. In symbols, then
\begin{equation}
  \label{eq:pears1}
  P(\Beta_{b,w}\le x) = P(\sum_{i=0}^{b+w-1} X_i \ge b) = P(\sum_{i=0}^{b+w-1} Y_i \le w-1),
\end{equation}
where the $X_i$ and $Y_i$ are independent $\Ber_x$ and $\Ber_{1-x}$
random variables, respectively. The last expression, with $x=1/2$,
establishes the following corollary:
\begin{cor} If a P\'olya urn is started with $b$ black and $w$ white
  balls, then the probability that the number of black and white balls
  will ever be equal is the same as the probability that in $b+w-1$
  tosses of a fair coin, no more than $w-1$ will be heads.
\end{cor}

Eq.~(\ref{eq:pears1}) was first derived computationally by Karl
Pearson in 1924~\cite{Pearson:1924vh}, for the purpose of expressing
sums of binomial coefficients in terms of the more easily computable
beta distribution. See also~\cite[p. 173]{feller_introduction_1968},
\cite[8.17.5]{olver_nist_2010}.

Both expressions for the equalization probability are useful. The
first, involving the beta function, is easily computed numerically.
The second can be used to establish central limit results, using the
deMoivre-Laplace theorem~\cite{feller_introduction_1968}, or large
deviations results, using Cram\'er's
theorem~\cite{kallenberg_foundations_2010}. This second expression can
also be written as an explicit sum of binomial coefficients,
\begin{displaymath}
  \frac{1}{2^{b+w-2}}\sum_{j=0}^{w-1}\binom{b+w-1}{j} = 
1-\frac{1}{2^{b+w-1}}\sum_{j=w}^{b-1}\binom{b+w-1}{j},
\end{displaymath}
and these forms are useful when either $w$ or $b-w$ are small,
respectively.

\section*{Acknowledgments}
\label{sec:ack}

I thank Eli Ben-Naim for helpful comments, and acknowledge support
from the Department of Energy under contract DE-AC52-06NA25396.

\end{document}